\documentclass[12pt,leqno]{article}  % 建檔日期 2013年7月12日
\usepackage{amsmath,amssymb,bbm,array,arydshln}
\usepackage{amsthm}
\usepackage{titlesec} % 重設章節標題所需的巨集
\usepackage[all,cmtip]{xy}

\titlespacing{\section}{0cm}{3.5pc}{1.5pc}

\makeatletter
\def\@citex[#1]#2{\if@filesw\immediate\write\@auxout{\string\citation{#2}}\fi
  \def\@citea{}\@cite{\@for\@citeb:=#2\do
    {\@citea\def\@citea{\@citesep}\@ifundefined
       {b@\@citeb}{{\bf ?}\@warning
       {Citation `\@citeb' on page \thepage \space undefined}}%
{\csname b@\@citeb\endcsname}}}{#1}}
\def\@citesep{; }
\makeatother

\newtheoremstyle{Kang}{}{}{\itshape}{}{\bf}{}{.5em}{}
\theoremstyle{Kang}
\newtheorem{theorem}{Theorem}[section]

\newtheorem{lemma}[theorem]{Lemma}

\newtheoremstyle{Kremark}{}{}{}{}{\bf}{}{.5em}{}
\theoremstyle{Kremark}
\newtheorem*{remark}{Remark.}
\newtheorem{defn}[theorem]{Definition}
\newtheorem{example}[theorem]{Example}
\newtheorem{other}{}

\newenvironment{Case}[1]{\medskip {\it Case #1.}}{}

\allowdisplaybreaks[1]  % 允許多行數學式分頁

\def\fn#1{\operatorname{#1}} % function work like \sin
\def\bm#1{\mathbbm{#1}}

\def\side#1#2{\mathop{^{#1}\mkern-1.5mu#2}}
\title{Invariants of wreath products and subgroups of $S_6$}
\author{Ming-chang Kang$^{(1)}$, Baoshan Wang$^{(2)}$ and Jian Zhou$^{(3)}$ \\[3mm]
\begin{minipage}{16cm} \begin{description} \itemsep=-1pt
\item[] $^{(1)}$Department of Mathematics, National Taiwan University, Taipei,\\ E-mail: kang@math.ntu.edu.tw
\item[] $^{(2)}$School of Mathematics and System Sciences, Beihang University,\\ Beijing, E-mail: bwang@buaa.edu.cn
\item[] $^{(3)}$School of Mathematical Sciences, Peking University, Beijing,\\ E-mail:
zhjn@math.pku.edu.cn
\end{description} \end{minipage}}
\date{}

\begin{document}

\maketitle

\footnote{\textit{\!\!\! $2010$ Mathematics Subject Classification}. 13A50, 14E08.}
\footnote{\textit{\!\!\! Keywords and phrases}.
Noether's problem, rationality problem, wreath products.}
\footnote{\!\!\! Parts of the work were finished when the second-named author and the third-named author were visiting National Taiwan University.}

\begin{abstract}
{\noindent\bf Abstract.} Let $G$ be a subgroup of $S_6$, the
symmetric group of degree 6. For any field $k$, $G$ acts naturally
on the rational function field $k(x_1,\ldots,x_6)$ via
$k$-automorphisms defined by $\sigma\cdot x_i=x_{\sigma(i)}$ for
any $\sigma\in G$, any $1\le i\le 6$. Theorem. The fixed field
$k(x_1,\ldots,x_6)^G$ is rational (=purely transcendental) over
$k$, except possibly when $G$ is isomorphic to $PSL_2(\bm{F}_5)$,
$PGL_2(\bm{F}_5)$ or $A_6$. When $G$ is isomorphic to
$PSL_2(\bm{F}_5)$ or $PGL_2(\bm{F}_5)$, then
$\bm{C}(x_1,\ldots,x_6)^G$ is $\bm{C}$-rational and
$k(x_1,\ldots,x_6)^G$ is stably $k$-rational for any field $k$.
The invariant theory of wreath products will be investigated also.
\end{abstract}

\newpage
%------------------------------------S1
\section{Introduction}

Let $k$ be a field. A finitely generated field extension $L$ of
$k$ is called $k$-rational if $L$ is purely transcendental over
$k$; it is called stably $k$-rational if $L(y_1,y_2,\ldots,y_m)$
is $k$-rational where $y_1,\ldots,y_m$ are elements which are
algebraically independent over $L$.

Let $G$ be a subgroup of $S_n$ where $S_n$ is the symmetric group
of degree $n$. For any field $k$, $G$ acts naturally on the
rational function field $k(x_1,\ldots,x_n)$ via $k$-automorphisms
defined by $\sigma\cdot x_i=x_{\sigma(i)}$ for any $\sigma\in G$,
any $1\le i\le n$. Noether's problem asks whether the fixed field
$k(x_1,\ldots,x_n)^G:=\{f\in k(x_1,\ldots,x_n):\sigma(f)=f$ for
all $\sigma\in G\}$ is $k$-rational (resp.\ stably $k$-rational)
\cite{No}. If $G$ is embedded in $S_N$ through the left regular
representation (where $N=|G|$), $k(x_1,\ldots,x_N)^G$ is nothing
but $k(V_{\fn{reg}})^G$ where $\rho:G\to GL(V_{\fn{reg}})$ is the
regular representation of $G$, i.e.\ $V_{\fn{reg}}=\bigoplus_{g\in
G} k\cdot e_g$ is a $k$-vector space and $h\cdot e_g=e_{hg}$ for
any $h,g\in G$. We will write $k(G)=k(V_{\fn{reg}})^G$ in the
sequel. The rationality problem of $k(G)$ is also called Noether's
problem, e.g.\ in the paper of Lenstra \cite{Le}.

If $G$ is a transitive subgroup of $S_n$, then the $G$-field
$k(x_1,\ldots,x_n)$ may be linearly embedded in the $G$-field
$k(V_{\fn{reg}})$ by Lemma \ref{l1.5}; thus $k(G)$ is rational
over $k(x_1,\ldots,x_n)^G$ by \cite[Theorem 2.1]{KW}. In
particular, if $k(x_1,\ldots,x_n)^G$ is $k$-rational, then so is
$k(G)$. We don't know whether the converse is true or not.

Noether's problem is related to the inverse Galois problem,
to the existence of generic $G$-Galois extensions,
and to the existence of versal $G$-torsors over $k$-rational field extensions.
For a survey of this problem, see \cite{GMS,Sa,Sw}.

This paper is a continuation of our paper \cite{KW}. We recall the
main results of \cite{KW} first. Let $k$ be any field, $G$ a
subgroup of $S_n$ acting naturally on $k(x_1,\ldots,x_n)$ by
$\sigma\cdot x_i=x_{\sigma(i)}$ for any $\sigma\in G$, any $1\le
i\le n$.

%-----------------------t1.1
\begin{theorem} \label{t1.1}
\begin{enumerate}
\item[{\rm (1)}] {\rm (\cite[Theorem 1.3]{KW})}
For any field $k$,
any subgroup $G$ of $S_n$, if $n\le 5$, then $k(x_1,\ldots,x_n)^G$ is $k$-rational.

\item[{\rm (2)}] {\rm (\cite[Theorem 1.4]{KW})}
Let $k$ be any field, $G$ be a transitive subgroup of $S_7$.
If $G$ is not isomorphic to the group $PSL_2(\bm{F}_7)$ or the group $A_7$,
then $k(x_1,\ldots,x_7)^G$ is $k$-rational.

Moreover, when $G$ is isomorphic to $PSL_2(\bm{F}_7)$ and $k$ is a field satisfying that $k\supset \bm{Q}(\sqrt{-7})$,
then $k(x_1,\ldots,x_7)^G$ is also $k$-rational.

\item[{\rm (3)}] {\rm (\cite[Theorem 1.5]{KW})}
Let $k$ be any field, $G$ be a transitive solvable subgroup of $S_{11}$.
Then $k(x_1,\ldots,x_{11})^G$ is $k$-rational.
\end{enumerate}
\end{theorem}

What we will prove in this paper is the case $G\subset S_6$.
We will establish the following theorem.

%--------------------t1.2
\begin{theorem} \label{t1.2}
Let $k$ be any field, $G$ be any subgroup of $S_6$. Then
$k(x_1,\ldots,x_6)^G$ is $k$-rational, except when $G$ is
isomorphic to the group $A_6$, $PSL_2(\bm{F}_5)$ or
$PGL_2(\bm{F}_5)$.

When $G$ is conjugate to the group $PSL_2(\bm{F}_5)$ or
$PGL_2(\bm{F}_5)$ embedded in $S_6$, then
$\bm{C}(x_1,\ldots,x_6)^G$ is $\bm{C}$-rational and
$k(x_1,\ldots,x_6)^G$ is stably $k$-rational for any field $k$.
\end{theorem}

First of all, note that we don't know whether
$k(x_1,\ldots,x_6)^{A_6}$ is $k$-rational or not. A second remark
is that, as abstract groups, $PSL_2(\bm{F}_5)$ (resp.\
$PGL_2(\bm{F}_5)$) is isomorphic to $A_5$ (resp.\ $S_5$). However,
the group $PGL_2(\bm{F}_5)$ embedded in $S_6$ as a transitive
subgroup (see the second paragraph of Section 3) provides a
6-dimensional reducible representation of $S_5$.

Since many transitive subgroups of $S_6$ are of the forms of
wreath products $H\wr G$ where $H\subset S_2$, $G\subset S_3$ or
$H\subset S_3$, $G\subset S_2$ (see Section 3), we embark on a
study of invariant theory of wreath products in Section 2 before
the proof of Theorem \ref{t1.2}. Here is a convenient criterion
for group actions of wreath products.

%----------------------t1.3
\begin{theorem} \label{t1.3}
Let $k$ be any field, $G\subset S_m$, $H\subset S_n$. Then the
wreath product $\widetilde{G}:=\allowbreak H\wr G$ can be regarded
as a subgroup of $S_{mn}$. If $k(x_1,\ldots,x_m)^G$ and
$k(y_1,\ldots,y_n)^H$ are $k$-rational, then
$k(z_1,\ldots,z_{mn})^{\widetilde{G}}$ is also $k$-rational.
\end{theorem}

An application of the above theorem is the following theorem; note
that our proof is different from the original proof of Tsunogai
\cite{Ts}.

%--------------------t1.4
\begin{theorem}[Tsunogai \cite{Ts}] \label{t1.4}
Let $k$ be any field, $p$ be a prime number.
For any integer $n\ge 2$,
let $P$ be a $p$-Sylow subgroup of $S_n$.
If $k(C_p)$ is $k$-rational,
then $k(x_1,\ldots,x_n)^P$ is also $k$-rational.
\end{theorem}

The following lemma helps to clarify the relationship of
rationality of $k(x_1,\ldots,x_n)^G$ and $k(G)$ when $G$ is a
transitive subgroup of $S_n$.

%-------------------------l1.5
\begin{lemma} \label{l1.5}
Suppose that $G$ is a transitive subgroup of $S_n$ acting
naturally on the rational function field $k(x_1, \cdots, x_n)$.
Let $G\to GL(V_{\fn{reg}})$ be the regular representation over a
field $k$, and $\{x(g):g\in G\}$ be a dual basis of
$V_{\fn{reg}}$. Then there is a $G$-equivariant embedding
$\Phi:\bigoplus_{1\le i\le n} k\cdot x_i \to \bigoplus_{g\in G}
k\cdot x(g)$.  In particular, $k(G)$ is rational over
$k(x_1,\ldots,x_n)^G$.
\end{lemma}

\begin{proof}
Note that $k(V_{\fn{reg}})=k(x(g):g\in G)$ with $h \cdot
x(g)=x(hg)$ for any $h,g \in G$.

Define $H=\{g \in G: g(1)=1 \}$. Choose a coset decomposition
$G=\cup_{1 \le i \le n}g_iH$ such that, for any $g \in G$, $g
\cdot g_iH=g_jH$ if and only if $g(i)=j$.

Define a $k$-linear map $\Phi:\bigoplus_{1\le i\le n} k\cdot g_i
H\to \bigoplus_{g\in G} k\cdot x(g)$ by $\Phi(g_i H)=\sum_{h\in H}
x(g_i h)$ $\in \bigoplus_{g\in G} k\cdot x(g)$.

$\Phi$ is a $G$-equivariant map. It is not difficult to show that
$\Phi$ is injective.

Consider the action of $G$ on the field $k(x_1, \cdots, x_n)$.
Identify the cosets $g_iH$ with $x_i$. It follows that, via
$\Phi$, the $G$-field $k(x_1,\ldots,x_n)$ is linearly embedded
into $k(x(g):\allowbreak g\in G)$. By applying Part (1) of
\cite[Theorem 2.1]{KW}, we find that $k(G)$ is rational over
$k(x_1,\ldots,x_n)^G$.
\end{proof}

\begin{remark}
If $G$ is a subgroup of $S_n$ and it is possible to embed
$\bigoplus_{1\le i\le n} k\cdot x_i$ into $\bigoplus_{g\in G}
k\cdot x(g)$, then $G$ is a transitive subgroup.

For, suppose that $T_1, \cdots, T_t$ are the $G$-orbits of the set
$\{x_1, \cdots, x_n \}$ with $t \ge 2$.  Each $T_i$ contributes a
trivial representation of $G$, but the regular representation
contains only one trivial representation. Thus it is impossible
that such a $G$-embedding exists.

\end{remark}

One may consider the rationality problem of $k(x_1,\ldots,x_8)^G$
where $G$ is a transitive subgroup of $S_8$. If $k$ contains
enough roots of unity, e.g. $\zeta_8 \in k$, then it is not very
difficult to show that $k(x_1,\ldots,x_8)^G$ is $k$-rational for
many such groups $G$ by standard methods and previously known
results (except possibly when $G$ is a non-abelian simple group).
However, if $k$ is any field and $g=\langle \sigma \rangle$ where
$\sigma =(1,2, \cdots,8)$, by Endo-Miyata's Theorem, it is known
that $k(x_1,\ldots,x_8)^G$ is $k$-rational if and only if
$k(\zeta_8)$ is cyclic over $k$ or char $k=2$ (\cite[Corollary
3.10; Le]{EM}; see also \cite[Theorem 1.8]{Ka}). The proof of the
above result is non-trivial; similar complicated situations may
happen in other subgroups of $S_8$.

\medskip
We organize this article as follows. A detailed discussion of
wreath products will be given in Section 2. Our method is
applicable not only in the Noether problem (i.e.\ the rational
invariants), but also in the polynomial invariants (see Theorem
\ref{t2.6}). The proof of Theorem \ref{t1.2} will be given in
Section 3.

Standing terminology. Throughout the paper, we denote by $S_n$,
$A_n$, $C_n$, $D_n$ the symmetric group of degree $n$, the
alternating group of degree $n$, the cyclic group of order $n$,
and the dihedral group of order $2n$ respectively. If $k$ is any
field, $k(x_1,\ldots,x_m)$ denotes the rational function field of
$m$ variables over $k$; similarly for $k(y_1,\ldots,y_n)$ and
$k(z_1,\ldots,z_l)$. When $\rho:G\to GL(V)$ is a representation of
$G$ over a field $k$, then $k(V)$ denotes the rational function
field $k(x_1,\ldots,x_n)$ with the induced action of $G$ where
$\{x_1,\ldots,x_n\}$ is a basis of the dual space $V^*$ of $V$. In
particular, when $V=V_{\fn{reg}}$ is the regular representation
space, denote by $\{x(g):g\in G\}$ a dual basis of $V_{\fn{reg}}$;
then $k(V_{\fn{reg}})=k(x(g):g\in G)$ where $h\cdot x(g)=x(hg)$
for any $h,g\in G$. We will write $k(G):=k(V_{\fn{reg}})^G$. When
$G$ is a subgroup of $S_n$, we say that $G$ acts naturally on the
rational function field $k(x_1,\ldots,x_n)$ by $k$-automorphisms
if $\sigma\cdot x_i=x_{\sigma(i)}$ for any $\sigma\in G$, any
$1\le i\le n$.

If $\sigma$ is a $k$-automorphism of the rational function field
$k(x_1,\ldots,x_n)$, it is called a monomial automorphism if
$\sigma (x_j)=b_j(\sigma) \prod_{1 \le i \le n} x_i^{a_{ij}}$
where $(a_{ij})_{1 \le i, j \le n} \in GL_n(\bm{Z})$, and
$b_j(\sigma) \in k^{\times}$. If $b_j(\sigma)=1$, the automorphism
$\sigma$ is called purely monomial. The group action of a finite
group $G$ acting on $k(x_1,\ldots,x_n)$ is called a monomial
action (resp. a purely monomial action) if $\sigma$ acts on
$k(x_1,\ldots,x_n)$ by a monomial (resp. purely monomial)
$k$-automorphism for all $\sigma \in G$.

In discussing wreath products, we denote by $X$ or $Y$ any set
without extra structures unless otherwise specified. The set $X_m$
is a finite set of $m$ elements; thus we write
$X_m=\{1,2,\ldots,m\}$. Similarly we write $Y_n=\{1,2,\ldots,n\}$.

As mentioned before, this article is a continuation of \cite{KW}.
Thus we will cite repeatedly the rationality criteria listed in
Section 2 of \cite{KW}.

%------------------------------------S2
\section{Wreath products}

Recall the definition of wreath products $H\wr G$ (or more
precisely $H\wr_X G$) in \cite[page 32, page 313; Is, page 73; DM,
pages 45--50]{Ro}.

%----------------------d2.1
\begin{defn} \label{d2.1}
Let $G$ and $H$ be groups and $G$ act on a set $X$ from the left such that $(g_1g_2)\cdot x=g_1\cdot (g_2\cdot x)$,
$1\cdot x=x$ for any $x\in X$, any $g_1,g_2\in G$.
Let $A$ be the set of all functions from $X$ to $H$;
$A$ is a group by defining $\alpha\cdot \beta(x):=\alpha(x)\cdot \beta(x)$ for any $\alpha,\beta\in A$, any $x\in X$.

In case $X$ is a finite set and $|X|=m$,
we will write $X=X_m=\{1,2,\ldots,m\}$ and $A=\prod_{1\le i\le m} H_i$.
Elements in $\prod_{1\le i\le m}H_i$ are of the form $\alpha=(\alpha_1,\ldots,\alpha_m)$ where each $\alpha_i\in H$
and $\alpha=(\alpha_1,\ldots,\alpha_m)$ corresponds to the element $\alpha\in A$ satisfying $\alpha(i)=\alpha_i$ for $1\le i\le m$.

The group $G$ acts on $A$ by
$(\side{g}{\alpha})(x)=\alpha(g^{-1}\cdot x)$ for any $g\in G$,
$\alpha \in A$, $x\in X$. It is easy to verify that
$\side{g_1g_2}{\alpha}=\side{g_1}{(\side{g_2}{\alpha})}$ for any
$g_1,g_2\in G$.

In case $X=X_m$ and $\alpha=(\alpha_1,\ldots,\alpha_m)$, then
$\side{g}{\alpha}=(\alpha_{g^{-1}(1)},\alpha_{g^{-1}(2)},\ldots,\alpha_{g^{-1}(m)})$
where we write $g(i)=g\cdot i$ for any $g\in G$, any $i\in X_m$.

The wreath product $H\wr_X G$ is the semi-direct product $A\rtimes G$ defined by
$(\alpha;g_1)\cdot (\beta;g_2)=(\alpha\cdot \side{g_1}{\beta};g_1g_2)$ for any $\alpha,\beta\in A$, any $g_1,g_2\in G$.

Sometimes we will write $H\wr G$ for $H\wr_X G$ if the set $X$ is understood from the context.
In particular, if $X=G$ and $G$ acts on $X$ by the left regular representation.

Since $G$ and $A$ may be identified as subgroups of $A\rtimes
G=H\wr_X G$, we will identify $g\in G$ and $\alpha \in A$ as
elements in $H\wr_X G$.
\end{defn}

%--------------------------d2.2
\begin{defn} \label{d2.2}
Let $G$ and $H$ be groups acting on the sets $X$ and $Y$ from the left respectively.
Then the wreath product $H\wr_X G$ acts on the set $Y\times X$ by defining
\[
(\alpha; g)\cdot (y,x)=((\alpha(g(x)))(y),g(x))
\]
for any $x\in X$, $y\in Y$, $g\in G$, $\alpha\in A$.
It is routine to verify that $((\alpha;g_1)\cdot(\beta;g_2))\cdot(y,x)=(\alpha;g_1)\cdot ((\beta;g_2)\cdot (y,x))$ for any $x\in X$,
$y\in Y$, $\alpha,\beta \in A$, $g_1,g_2\in G$.
\end{defn}

In case $G\subset S_m$, $H\subset S_n$,
we may regard $H\wr_{X_m} G$ as a subgroup of $S_{mn}$ because $H\wr_{X_m} G$ acts faithfully on the set
$Y_n\times X_m=\{(j,i):1\le i\le m,1\le j\le n\}$.

If $Y$ is the polynomial ring $k[y_1,\ldots,y_n]$ over a field
$k$, we require the action of $H$ on $Y$ satisfies an extra
condition that, for any $h\in H$, the map $\phi_h:f\mapsto h\cdot
f$ is a $k$-algebra morphism where $f\in k[y_1,\ldots,y_n]$.

%-----------------------ex2.3
\begin{example} \label{ex2.3}
Let $G\subset S_m$, $H\subset S_n$. Then $G$ acts on
$X_m=\{1,2,\ldots,m\}$ and $H$ acts on $Y_n=\{1,2,\ldots,n\}$.
Thus $H\wr_{X_m}G$ acts faithfully on $Y_n\times X_m$ by,
\[
(\alpha_1,\alpha_2,\ldots,\alpha_m;g)\cdot (j,i)=(\alpha_{g(i)}
(j),g(i))
\]
where $\alpha=(\alpha_1,\ldots,\alpha_m)\in A$, $g\in G$.
\end{example}

For $1\le l\le m$, $h\in H$, define $\alpha^{(l)}(h)\in A$ by $(\alpha^{(l)}(h))(l)=h\in H$ and $(\alpha^{(l)}(h))(i)=1\in H$ if $i\ne l$.
It is clear that $H\wr_{X_m} G=\langle \alpha^{(l)}(h), g:1\le l\le m,g\in G,h\in H\rangle$.

In case $G$ is a transitive subgroup of $S_m$,
it is not difficult to verify that $H\wr_{X_m}G=\langle \alpha^{(1)}(h), g:g\in G,h\in H\rangle$.
Note that
\begin{align*}
\alpha^{(1)}(h) &: (j,i)\mapsto \begin{cases} (j,i) & \text{if }i\ne 1, \\ (h(j),1) & \text{if }i=1, \end{cases} \\
g &: (j,i)\mapsto (j,g(i)).
\end{align*}

%--------------------------ex2.4
\begin{example} \label{ex2.4}
Let $p$ be a prime number and $G,H\subset S_p$.
We denote $\lambda=(1,2,\ldots,p)\in S_p$ and identify the set $Y_p\times X_p$ with the set $X_{p^2}$ by the function
\begin{align*}
\varphi:{} & Y_p\times X_p\to X_{p^2} \\
& (j,i)\mapsto i+jp
\end{align*}
where the elements in $X_{p^2}$ are taken modulo $p^2$.

Let $G=\langle \lambda\rangle$, $H=\langle \lambda\rangle$ act
naturally on $X_p$ and $Y_p$ respectively. Then $H\wr_{X_p} G$ is
a group of order $p^{1+p}$ acting on $X_{p^2}$ by identifying
$\sigma=\alpha^{(1)}(\lambda)$ and $\tau=\lambda$ with
\begin{multline*}
(1,1+p,1+2p,\ldots,1+(p-1)p) \text{ and } (1,2,\ldots,p)(p+1,p+2,\ldots,2p)\cdots ((p-1)p+1, \\
(p-1)p+2,\ldots,(p-1)p+p)
\end{multline*}
in $S_{p^2}$.
$H\wr_{X_p} G$ is a $p$-Sylow subgroup of $S_{p^2}$.

Inductively, let $P_r$ be a $p$-Sylow subgroup of $S_{p^r}$
constructed above. Let $H=\langle\lambda\rangle\subset S_p$,
$G=P_r\subset S_{p^r}$. Then $H\wr_{X_{p^r}} G$ is a group of order
$p^{1+p+p^2+\cdots+p^r}$ acting on $X_{p^{r+1}}$ (by the function
$\varphi: Y_p\times X_{p^r} \to X_{p^{r+1}}$ defined by
$\varphi(j,i)=i+j\cdot p^r)$. Thus $H\wr_{X_{p^r}} G$ is a $p$-Sylow
subgroup of $S_{p^{r+1}}$ \cite[page 49]{DM}.

If $n$ is a positive integer and write $n=n_0+n_1p+n_2p^2+\cdots+n_tp^t$ where $0\le n_i\le p-1$ and $n<p^{t+1}$,
then a $p$-Sylow subgroup of $S_n$ is isomorphic to
\[
(P_1)^{n_1} \times \cdots\times (P_t)^{n_t}
\]
where each $P_i$ is isomorphic to a $p$-Sylow subgroup of $S_{p^i}$ for $1\le i\le t$.
\end{example}

We reformulate Theorem \ref{t1.3} as the following theorem.

%------------------------t2.5
\begin{theorem} \label{t2.5}
Let $k$ be any field, $G\subset S_m$ and $H\subset S_n$.
Let $G$ and $H$ act on the rational function fields $k(x_1,\ldots,x_m)$ and $k(y_1,\ldots,y_n)$ respectively via $k$-automorphisms defined
by $g\cdot x_i=x_{g(i)}$, $h\cdot y_j=y_{h(j)}$ for any $g\in G$, $h\in H$, $1\le i\le m$, $1\le j\le n$.
Then $\widetilde{G}:= H\wr_{X_m} G$ may be regarded as a subgroup of $S_{mn}$ acting on the rational function field
$k(x_{ij}:1\le i\le m, 1\le j\le n)$ by Definition \ref{d2.2}.
Assume that both $k(x_1,\ldots,x_m)^G$ and $k(y_1,\ldots,y_n)^H$ are $k$-rational.
Then $k(x_{ij}:1\le i\le m,1\le j\le n)^{\widetilde{G}}$ is also $k$-rational.
\end{theorem}

\begin{proof}
Adopt the notations in Example \ref{ex2.3}.
For any $1\le l\le m$, any $h\in H$,
define $\alpha^{(l)}(h)\in \widetilde{G}=H\wr_{X_m} G$.
Note that $A=\langle \alpha^{(l)} (h):1\le l\le m,h\in H\rangle$.
Then we find that, for any $g\in G$, any $\alpha^{(l)}(h)$, the actions are given by
\begin{align*}
g &: x_{ij} \mapsto x_{g(i),j} \\
\alpha^{(l)} (h) &: x_{ij} \mapsto \begin{cases} x_{ij} & \text{if }i\ne l, \\ x_{l,h(j)} & \text{if }i=l, \end{cases}
\end{align*}
where $1\le i\le m$, $1\le j\le n$.

Since $k(y_1,\ldots,y_n)^H$ is $k$-rational,
we may find $F_1(y),\ldots,F_n(y)\in k(y_1,\ldots,y_n)$ such that $F_j(y)=F_j(y_1,\ldots,y_n)\in k(y_1,\ldots,y_n)$ for $1\le j\le n$,
and $k(y_1,\ldots,y_n)^H=k(F_1(y),\ldots,F_n(y))$.
It follows that $k(x_{ij}:1\le j\le n)^{\langle \alpha^{(i)}(h): h\in H\rangle}=k(F_1(x_{i1},\ldots,x_{in}),\allowbreak
F_2(x_{i1},\ldots,x_{in}),\ldots,F_n(x_{i1},\ldots,x_{in}))$.
Hence $k(x_{ij}: 1\le i\le m, 1\le j\le n)^A=k(F_1(x_{i1},\ldots,x_{in}),\ldots,F_n(x_{i1},\ldots,x_{in}):1\le i\le m)$.

Note that $F_1(x_{i1},\ldots,x_{in}),\ldots,F_n(x_{i1},\ldots,x_{in})$ (where $1\le i\le m$) are algebraically independent over $k$
and $g\cdot F_j(x_{i1},\ldots,x_{in})=F_j(x_{g(i),1},\ldots,x_{g(i),n})$ for any $g\in G$.
Denote $E_{ij}=F_j(x_{i1},\ldots,x_{in})$ for $1\le i\le m$, $1\le j\le n$.
We find that $k(E_{ij}:1\le i\le m,1\le j\le n)$ is a rational function field over $k$ with $G$-actions given by
$g\cdot E_{ij}=E_{g(i),j}$ for any $g\in G$.

It follows that $k(x_{ij}:1\le i\le m,1\le j\le n)^{\widetilde{G}}=\{k(x_{ij}:1\le i\le m,1\le j\le n)^A\}^G=
k(E_{ij}:1\le i\le m,1\le j\le n)^G=k(E_{11},E_{21},\ldots,E_{m,1})^G (t_{ij}:2\le i\le m, 1\le j\le n)$ for some $t_{ij}$ satisfying
that $g(t_{ij})=t_{ij}$ for any $g\in G$ by applying \cite[Theorem 2.1]{KW}.

Since $k(x_1,\ldots,x_m)^G$ is $k$-rational,
it follows that $k(E_{11},E_{21},\ldots,E_{m,1})^G$ is also $k$-rational.
Hence the result.
\end{proof}

\begin{proof}[Proof of Theorem \ref{t1.4}] -----------------

\medskip
Let $P$ be a $p$-Sylow subgroup of $S_n$. We will show that, if
$k(C_p)$ is $k$-rational, then $k(x_1,\ldots,x_n)^P$ is also
$k$-rational.

Without loss of generality,
we may assume that $P$ is the $p$-Sylow subgroup constructed in Example \ref{ex2.4}.

\medskip
Step 1. Consider the case $n=p^t$ first. Then the $p$-Sylow
subgroup $P_t$ is of the form $P_t=H\wr_{X_{p^{t-1}}} G$ where
$H=\langle\lambda\rangle \simeq C_p$ (with
$\lambda=(1,2,\ldots,p)\in S_p$), and $G=P_{t-1}\subset
S_{p^{t-1}}$. Note that
$k(y_1,\ldots,y_p)^H=k(y_1,\ldots,y_p)^{\langle\lambda\rangle}\simeq
k(C_p)$. By induction, $k(x_1,\ldots,x_{p^{t-1}})^G$ is also
$k$-rational. Applying Theorem \ref{t2.5}, it follows that
$k(z_1,\ldots,z_{p^t})^{P_t}$ is $k$-rational.

\bigskip
Step 2. Suppose that $G_1\subset S_m$ and $G_2\subset S_n$. Thus
$G_1$ acts on $k(x_1,\ldots,x_m)$ and $G_2$ acts on
$k(y_1,\ldots,y_n)$. If $k(x_1,\ldots,x_m)^{G_1}$ and
$k(y_1,\ldots,y_n)^{G_2}$ are $k$-rational, then
$k(x_1,\ldots,x_m)^{G_1}=k(F_1,\ldots,F_m)$ and
$k(y_1,\ldots,y_n)^{G_2}=k(F_{m+1},\ldots,F_{m+n})$ where
$F_i=F_i(x_1,\ldots,x_m)$ for $1\le i\le m$, and
$F_{m+j}=F_{m+j}(y_1,\ldots,y_n)$ for $1\le j\le n$. It follows
that $k(x_1,\ldots,x_m,$ $y_1,\ldots,y_n)^{G_1\times
G_2}=k(F_1,F_2,\ldots,F_{m+n})$, because $[k(x_1,\ldots,x_m,y_1,$
$\ldots,y_n):k(F_1,\ldots,F_{m+n})] =
[k(x_1,\ldots,x_m,y_1,\ldots,y_n):k(F_1,\ldots,F_m,\allowbreak
y_1,\ldots,y_n)] \cdot [k(F_1,$ $\ldots,F_m,y_1,\ldots,y_n):
k(F_1,\ldots,F_{m+n})]=|G_1|\cdot |G_2|$. Thus
$k(x_1,\ldots,x_m,y_1,\ldots,$ $y_n)^{G_1\times G_2}$ is
$k$-rational

\bigskip
Step 3.
Consider the general case.
As in Example \ref{ex2.4},
write $n=n_0+n_1p+n_2p^2+\cdots+n_tp^t$ where $0\le n_i\le p-1$ and $n<p^{t+1}$.

By Step 1, $k(x_1,\ldots,x_{p^i})^{P_i}$ is $k$-rational for any
$1\le i\le t$. Thus $k(x_1,\ldots,x_{p^i},\allowbreak
x_{p^i+1},\ldots,x_{2p^i},\ldots,x_{n_i\cdot p^i})^{P_i^{n_i}}$ is
also $k$-rational by Step 2.

It follows that $k(x_1,\ldots,x_{n-n_0})^P$ is $k$-rational when $P=(P_1)^{n_1}\times (P_2)^{n_2}\times \cdots \times (P_t)^{n_t}$.
Thus $k(x_1,\ldots,x_n)^P$ is also $k$-rational.
\end{proof}

\begin{remark}
In \cite[Theorem 1.7]{KP}, it was proved that, if $k(G)$ and
$k(H)$ are $k$-rational, then so is $k(H\wr G)$ where the group
$H\wr G$ is actually $H\wr_{X} G$ with $X=G$ and $G$ acting on $X$
by the left regular representation.  We remark that this result
follows from Theorem \ref{t2.5} and Lemma \ref{l1.5} if $G$ and
$H$ are transitive subgroups of $S_m$ and $S_n$ respectively.

Similarly, it was proved that, if $k(G_1)$ and $k(G_2)$ are
$k$-rational, so is $k(G_1\times G_2)$ \cite[Theorem 1.3]{KP}.
This result may be generalized to representations other than the
regular representation as follows.
\end{remark}

\begin{theorem} \label{t2.9}
Let $G_1, G_2$ be finite groups, $G_1 \subset S_m$, $G_2 \subset
S_n$, and $G:=G_1 \times G_2$. Let $G$ act naturally on the
rational function field $k(z_{ij}: 1 \le i \le m, 1 \le j \le n)$
by $g_1 \cdot z_{ij} = z_{g_1(i),j}$, $g_2 \cdot z_{ij} =
z_{i,g_2(j)}$ for any $g_1 \in G_1$, any $g_2 \in G_2$. If both
$k(x_1, \cdots, x_m)^{G_1}$ and $k(y_1, \cdots, y_n)^{G_2}$ are
$k$-rational, then $k(z_{ij}: 1 \le i \le m, 1 \le j \le n)^G$ is
also $k$-rational.
\end{theorem}

\begin{proof}
Define an action of $G$ on the rational function field $k(x_i,
y_j: 1 \le i \le m, 1 \le j \le n)$ by $g_1 \cdot x_i=x_{g_1(i)}$,
$g_1 \cdot y_j=y_j$, $g_2 \cdot x_i=x_i$, $g_2 \cdot
y_j=y_{g_2(j)}$ for any $g_1 \in G_1$, any $g_2 \in G_2$.

The $k$-linear map $\Phi: (\oplus_{1 \le i \le m} k \cdot x_i)
\oplus (\oplus_{1 \le j \le n} k \cdot y_j) \longrightarrow
\oplus_{1 \le i \le m, 1 \le j \le n} k \cdot z_{ij}$ defined by
$\Phi(x_i)= \sum_{1 \le j \le n} z_{ij}$ and $\Phi(y_j)= \sum_{1
\le i \le m} z_{ij}$ is $G$-equivariant.

By Part (1) of \cite[Theorem 2.1]{KW}, $k(z_{ij}: 1 \le i \le m, 1
\le j \le n)^G$ is rational over $k(x_i, y_j: 1 \le i \le m, 1 \le
j \le n)^G$. It is easy to see that $k(x_i, y_j: 1 \le i \le m, 1
\le j \le n)^G$ is $k$-rational. So is $k(z_{ij}: 1 \le i \le m, 1
\le j \le n)^G$.
\end{proof}

It is easy to adapt the proof of the above theorem to the
following theorem.

\begin{theorem} \label{t2.10}
Let $G_1, G_2$ be finite groups, $G:=G_1 \times G_2$. Suppose that
$\rho_1:G_1 \to GL(V)$, $\rho_2:G_2 \to GL(W)$ are faithful
representations over a field $k$. Let $G$ act on $V\otimes_k W$ by
$g_1 \cdot (v \otimes w) =(g_1 \cdot v) \otimes w$, $g_2 \cdot (v
\otimes w) =v \otimes (g_2 \cdot w)$ for any $g_1 \in G_1$, $g_2
\in G_2$, $v \in V, w \in W$. Assume that (i) $V$ and $W$ contain
a trivial representation, and (ii) $k(V)^{G_1}$ and $k(W)^{G_2}$
are $k$-rational. Then $k(V\otimes_k W)^G$ is also $k$-rational.
\end{theorem}

\begin{proof}
Define a suitable action of $G$ on $V \oplus W$ as in the proof of
Theorem \ref{t2.9}.

Let $V^{\ast}$ and $W^{\ast}$ be the dual spaces of $V$ and $W$
respectively. Since $V$ contains a trivial representation, it is
possible to find a non-zero element $v_0 \in V^{\ast}$ such that
$g_1 \cdot v_0 =v_0$ for any $g_1 \in G_1$. Similarly, find a
non-zero element $w_0 \in W^{\ast}$ such that $g_2 \cdot w_0 =w_0$
for any $g_2 \in G_2$.

Define the embedding $\Phi: V^{\ast} \oplus
W^{\ast}\longrightarrow V^{\ast} \otimes_k W^{\ast}$ defined by
$\Phi(x)= x \otimes w_0$ and $\Phi(y)= v_0 \otimes y$. $\Phi$ is
$G$-equivariant. The remaining proof is omitted.

\end{proof}

\medskip
Now let's turn to the polynomial invariants of wreath products.

Suppose that a group $H$ acts on $Y$ which is a finitely generated
commutative algebra over a field $k$. In this case, we require
that, for any $h\in H$, the map $\varphi_h:Y\to Y$, defined by
$\varphi_h(y)=h\cdot y$ for any $y\in Y$, is a $k$-algebra
morphism. In Theorem \ref{t2.6}, we take $Y=k[y_j:1\le j\le n]$ a
polynomial ring; in that situation we require furthermore that
$\varphi_h(y_j)=\sum_{1\le l\le n} b_{lj} (h) y_l$ where
$(b_{lj}(h))_{1\le l,j\le n} \in GL_n (k)$.

The method presented in Theorem \ref{t2.6} is valid for a more
general setting, e.g.\ $H$ is a reductive group over a field $k$
and $Y$ is a finitely generated commutative $k$-algebra. In order
to highlight the crucial idea of our method, we choose to
formulate the results for some special cases only.

From now on till the end of this section, $G\subset S_m$ and $H$
is a finite group such that $G$ acts on $X_m=\{1,2,\ldots,m\}$ and
$H$ acts on the polynomial ring $k[y_j:1\le j\le n]$ over a field
$k$ and $\varphi_h(y_j)=\sum_{1\le l\le n} b_{lj} (h) y_l$ for any
$h\in H$, any $1\le j\le n$ with $(b_{lj} (h))_{1\le l,j\le n} \in
GL_n(k)$. Define $\widetilde{G}:=H\wr_{X_m} G$ which acts on the
polynomial ring $k[x_{ij}:1\le i\le m, 1\le j\le n]$ defined by

\begin{align*}
g &: x_{ij} \mapsto x_{g(i),j} \\
\alpha^{(l)}(h) &: x_{ij} \mapsto
\begin{cases} x_{ij} &, \text{ if } i\ne l, \\ \sum_{1\le t\le n} b_{tj}(h)x_{lt} &, \text{ if }i=l, \end{cases}
\end{align*}
where $g\in G$, $\alpha^{(l)} (h)\in A$.

The goal is to find the ring of invariants $k[x_{ij}:1\le i\le m,
1\le j\le n]^{\widetilde{G}}:= \{f\in k[x_{ij}]:\lambda(f)=f$ for
any $\lambda \in \widetilde{G}\}$.

%----------------------t2.6
\begin{theorem} \label{t2.6}
Let $k$ be a field, $\widetilde{G}:=H\wr_{X_m} G$ act on the
polynomial ring $k[x_{ij}:1\le i\le m,1\le j\le n]$ as above.
Assume that $\gcd\{|G|,\fn{char}k\}=1$. Suppose that
$k[y_1,\ldots,y_n]^H=k[F_1(y),\ldots,F_N(y)]$ where
$F_t(y)=F_t(y_1,\ldots,y_n)\in k[y_1,\ldots,y_n]$ for $1\le t\le
N$ $(N$ is some integer $\ge n)$. Define an action of $G$ on the
polynomial ring $k[X_{it}:1\le i\le m, 1\le t\le N]$ by
$g(X_{it})=X_{g(i),t}$ for any $g\in G$, $1\le i\le m$, $1\le t\le
N$. Define a $k$-algebra morphism
\[
\Phi: k[X_{it}:1\le i\le m,1\le t\le N] \to
k[F_1(x_{i1},\ldots,x_{in}),\ldots,F_N(x_{i1},\ldots,x_{in}):1 \le
i \le m]
\]
by $\Phi(X_{it})=F_t(x_{i1},\ldots,x_{in})$

If $k[X_{it}:1\le i\le m,1\le t\le N]^G=k[H_1(X),\ldots,H_M(X)]$
where $H_s(X)=H_s(X_{11},\ldots,X_{it},\ldots,X_{m,N}) \in
k[X_{it}: 1\le i\le m,1\le t\le N]$ for $1\le s\le M$, then
$k[x_{ij}:1\le i\le m,1\le j\le
n]^{\widetilde{G}}=k[\Phi(H_1(X)),\ldots,\Phi(H_M(X))]$.
\end{theorem}

\begin{remark}
Even without the assumption $\gcd\{|G|,\fn{char}k\}=1$, it is
still known that $k[x_{ij}: 1\le i\le m, 1\le j\le
n]^{\widetilde{G}}$ is finitely generated over $k$. With the
assumption $\gcd\{|G|,\fn{char}k\}=1$, the ring of invariants
$k[x_{ij}:1\le i\le m, 1\le j\le n]^{\widetilde{G}}$ can be
computed effectively (see Example \ref{ex2.7}).
\end{remark}

\begin{proof}
Step 1.
For $1\le l\le m$, define $H^{(l)}=\langle\alpha^{(l)} (h):h\in H\rangle$.
Then $A=\langle H^{(l)}:1\le l\le m\rangle$ and $k[x_{ij}:1\le i\le m, 1\le j\le n]^A=\bigcap_{1\le l\le m} k[x_{ij}:1\le i\le m, 1\le j\le n]^{H^{(l)}}$.

On the other hand, from the definition of $F_1(y),\ldots,F_N(y)$,
it is clear that $k[x_{ij}:1\le i\le m,1\le j\le
n]^{H^{(l)}}=k[F_1(x_{l1},\ldots,x_{ln}),\ldots,F_N(x_{l1},ldots,x_{ln})]
[x_{ij}:i\ne l,1\le i\le m,1\le j\le n]$.

It follows that $k[x_{ij}:1\le i\le m,1\le j\le n]^A=k[F_1(x_{i1},\ldots,x_{in}),\ldots,F_N(x_{i1},\ldots,\allowbreak x_{in}):1\le i\le m]$
and $G$ acts on it by
\[
g:F_t(x_{i1},\ldots,x_{in})\mapsto F_t(x_{g(i),1},\ldots,x_{g(i),n})
\]
where $g\in G$, $1\le t\le N$.

\bigskip
Step 2. It is clear that $\Phi$ is an equivariant $G$-map.

We claim that
$k[F_1(x_{i1},\ldots,x_{in}),\ldots,F_N(x_{i1},\ldots,x_{in})
:i\le m]^G=\Phi(k[X_{it}:1$ $\le i\le m, 1\le t\le N]^G)$.

For, if $h\in
k[F_1(x_{i1},\ldots,x_{in}),\ldots,F_N(x_{i1},\ldots,x_{in}):1\le
i\le m]^G$, choose a preimage $\tilde{h}$ of $h$, i.e.\
$\Phi(\tilde{h})=h$. Since $h=(\sum_{g\in G} g(h))/|G|$ because
$g(h)=h$ for any $g\in G$, it follows that
$h=\Phi(\tilde{h})=\Phi(\sum_{g\in G} g(\tilde{h}))/|G|$. Since
$\sum_{g\in G} g(\tilde{h})\in k[X_{it}: 1\le i\le m, 1\le t\le
N]^G$, it follows that $h$ belongs to the image of $k[X_{it}:1\le
i\le m, 1\le t\le N]^G$.
\end{proof}

%-------------------------ex2.7
\begin{example} \label{ex2.7}
Let $G=S_m$ act on the polynomial ring $k[X_{it}:1\le i\le m, 1\le t\le N]$ by $g(X_{it})=X_{g(i),t}$ for any $g\in G$,
any $1\le i\le m$, $1\le t\le N$.

Let $f_1,\ldots,f_m$ be the elementary symmetric functions of $X_1,\ldots,X_m$,
i.e.\ $f_1=\sum_{1\le i\le m} X_i$, $f_2=\sum_{1\le i< j\le m} X_i X_j$, $\ldots$, $f_m=X_1X_2\cdots X_m$.

The polarized polynomials of $f_2$ with respect to the variables $X_{i1}$ and $X_{i2}$ where $1\le i\le m$ are
\[
\sum_{1\le i<j\le m} X_{i1}X_{j1}, \quad \sum_{1\le i,j\le m \atop i\ne j} X_{i1}X_{j2}, \quad
\sum_{1\le i<j\le m} X_{i2}X_{j2}.
\]

Similarly we may define the polarized polynomials of $f_2$ with
respect to the variables $X_{i1},X_{i2},X_{i3},\ldots, X_{im}$
where $1\le i\le N$. See \cite[pages 60--61]{Sm} for details.

Assume that $1/|G|! \in k$.
Then $k[X_{it}:1\le i\le m, 1\le t\le N]^{S_m}$ is generated over $k$ by all the polarized polynomials of $f_1,\ldots,f_m$
(see \cite[page 68, Theorem 3.4.1]{Sm}).

If we assume only that $1/|G|\in k$ and $G\subset GL_m (k)$ is a
finite group, it is still possible to compute $k[X_{it}:1\le i\le
m, 1\le t\le N]^G$ effectively. See \cite{Fl,Fo} for
details.
\end{example}

%------------------------------ex2.8
\begin{example} \label{ex2.8}
Let $\sigma=(1,2)$, $G=\langle \sigma\rangle$ acting on
$X_2=\{1,2\}$. Also let $H=\langle \tau\rangle\simeq C_3$ acting
on $\bm{C}[y_1,y_2,y_3]$ by $\tau: y_1\mapsto y_1$, $y_2\mapsto
\omega y_2$, $y_3\mapsto \omega^2 y_3$ where
$\omega=e^{2\pi\sqrt{-1}/3}$. Define $\widetilde{G}=H\wr_{X_2} G$
and let it act on $\bm{C}[x_{ij}:1\le i\le 2, 1\le j\le 3]$ by
\begin{align*}
\sigma &: x_{ij}\mapsto x_{\sigma(i),j} \\
\tau_1 &: x_{11}\mapsto x_{11},~ x_{12}\mapsto\omega x_{12},~ x_{13}\mapsto \omega^2 x_{13},~ x_{2j}\mapsto x_{2j}, \\
\tau_2 &: x_{21}\mapsto x_{21},~ x_{22}\mapsto\omega x_{22},~ x_{23}\mapsto \omega^2 x_{23},~ x_{1j}\mapsto x_{1j}.
\end{align*}

Then $\widetilde{G}=\langle\sigma,\tau_1,\tau_2\rangle$ and $\bm{C}[x_{ij}:1\le i\le 2, 1\le j\le 3]^{\langle \tau_1,\tau_2\rangle}
=\bm{C}[x_{11},f_1,f_2,f_3,x_{21},\allowbreak f'_1,f'_2,f'_3]$ where $f_1=x_{12}^3$, $f_2=x_{12}x_{13}$, $f_3=x_{13}^3$,
$f'_1=x_{22}^3$, $f'_2=x_{22}x_{23}$, $f'_3=x_{23}^3$.
Moreover, $\sigma: x_{11}\leftrightarrow x_{21}$, $f_1\leftrightarrow f'_1$, $f_2\leftrightarrow f'_2$, $f_3\leftrightarrow f'_3$.

Define $X_{it}$ (where $1\le i\le 2$, $1\le t\le 4$) as in Theorem
\ref{t2.6} with $\sigma:X_{ij}\leftrightarrow X_{\sigma(i),j}$. It
is easy to verify that $\bm{C}[x_{ij}:1\le i\le 2, 1\le j\le
3]^{\widetilde{G}}=\bm{C}[x_{11}+x_{21},f_1+f'_1,f_2+f'_2,\allowbreak
f_3+f'_3,x_{11}x_{21},
f_1f'_1,f_2f'_2,f_3f'_3,x_{11}f'_1+x_{21}f_1,x_{11}f'_2+x_{21}f_2,x_{11}f'_3+x_{21}f_3,f_1f'_2+f'_1f_2,\allowbreak
f_1f'_3+f'_1f_3,f_2f'_3+f'_2f_3]$.
\end{example}

%--------------------------------------------S3
\section{Proof of Theorem \ref{t1.2}}

Let $G$ be a subgroup of $S_6$ acting naturally on
$k(x_1,\ldots,x_6)$. We will study the rationality of
$k(x_1,\ldots,x_6)^G$.

As in the first paragraph in the proof of Theorem 3.4 of \cite{KW},
we may assume that $G$ is a transitive subgroup without loss of generality.
According to \cite[page 60]{DM},
such a group is conjugate to one of the following 16 groups: \par
\begin{align*}
G_1 &=\langle (1,2,3,4,5,6)\rangle \simeq C_6, \\
G_2 &=\langle (1,2)(3,4)(5,6),(1,3,5)(2,6,4)\rangle \simeq S_3, \\%(1,3,5)(2,4,6)
G_3 &=\langle (1,2,3,4,5,6),(1,6)(2,5)(3,4)\rangle \simeq D_6, \\
G_4 &=\langle (1,2,3)(4,5,6),(1,2)(4,5),(1,4)\rangle \simeq S_2 \wr_{X_3} S_3, \\
G_5 &=\langle (1,2,3)(4,5,6),(1,2)(4,5),(1,4)(2,5)\rangle=G_4\cap A_6, \\
G_6 &=\langle (1,2,3)(4,5,6),(1,5,4,2)\rangle \simeq S_4, \\
G_7 &=\langle (1,2,3)(4,5,6),(1,4)(2,5)\rangle =G_6\cap A_6, \\
G_8 &=\langle (1,2,3)(4,5,6),(1,4)\rangle \simeq C_2\wr_{X_3} C_3, \\
G_9 &=\langle (1,2,3),(1,2),(1,4)(2,5)(3,6)\rangle \simeq S_3\wr_{X_2} C_2, \\
G_{10} &=\langle (1,2,3),(1,4,2,5)(3,6),(1,2)(4,5)\rangle=G_9 \cap A_6, \\%(1,5,4,2)(3,6)
G_{11} &=\langle (1,2,3),(1,2)(4,5),(1,4)(2,5)(3,6)\rangle \simeq C_3^2 \rtimes C_2^2, \\
G_{12} &=\langle (1,2,3),(1,4)(2,5)(3,6)\rangle \simeq C_3 \wr_{X_2} C_2, \\
G_{13} &=\langle (0,1,2,3,4),(0,\infty)(1,4)(1,2,4,3)\rangle \simeq PGL_2(\bm{F}_5), \\
G_{14} &=\langle (0,1,2,3,4),(0,\infty)(1,4)\rangle \simeq PSL_2(\bm{F}_5), \\
G_{15} &=A_6, \\
G_{16} &=S_6.
\end{align*}

Be aware that the above descriptions of the groups $G_2$, $G_4$
and $G_{10}$ are different from those in \cite[page 60]{DM},
because the presentation there contains some minor mistakes.

Note that, the rationality of $k(x_1,\ldots,x_6)^{G_{16}}$ is
easy. On the other hand, the rationality of
$k(x_1,\ldots,x_6)^{G_{15}}$ is still an open problem. When
$G=G_9$, $G_{10}$, $G_{11}$ or $G_{12}$, the rationality of
$k(x_1,\ldots,x_6)^G$ was proved in \cite[Section 3]{Zh}.

When $G=G_4$, $G_8$, $G_9$ or $G_{12}$, the group is a wreath
product. We may apply Theorem \ref{t2.5}, because
$k(x_1,x_2,x_3)^{S_3}$, $k(C_2)$, $k(C_3)$ are $k$-rational by
Theorem \ref{t1.1}. For example, consider the case $G=G_4$. Note
that $S_3$ acts transitively on $X_3=\{1,2,3\}$. Define
$\widetilde{G}=S_2\wr_{X_3} S_3$, $G=S_3$,
$H=S_2=\langle\tau\rangle$ acts on $Y_2=\{1,2\}$. In the notation
of Section 2, we have $A=\prod_{1\le i\le 3}H_i$ where each
$H_i=H$. It follows that
$\widetilde{G}=\langle\sigma_1,\sigma_2,\alpha^{(1)}(\tau)\rangle$
where $\sigma_1=(1,2,3)$, $\sigma_2=(1,2)\in G$ by Example
\ref{ex2.3}. It is not difficult to show that $\widetilde{G}\simeq
G_4$.

Thus it remains to study the rationality of $k(x_1,\ldots,x_6)^G$ when $G=G_1$, $G_2$, $G_3$, $G_5$, $G_6$, $G_7$, $G_{13}$ and $G_{14}$.
We study the case $G_{13}$ and $G_{14}$ first.

%-------------------------t3.1
\begin{theorem} \label{t3.1}
If $G=G_{13}$ or $G_{14}$, then $\bm{C}(x_1,\ldots,x_6)^G$ is
$\bm{C}$-rational, and $k(x_1,\ldots,x_6)^G$ is stably
$k$-rational where $k$ is any field.
\end{theorem}

\begin{proof}
We will prove the $G_{13}$ is isomorphic to $S_5$ as abstract
groups. Then it will be shown that the permutation representation
of $G_{13}$ as a subgroup of $S_6$ is equivalent to the direct sum
of the trivial representation and a 5-dimensional irreducible
representation of $S_5$ over $\bm{Q}$. Then we will apply the
results of \cite{Sh}.

\medskip
Step 1. Since $G_{13}=PGL_2(\bm{F}_5)$ is the automorphism group
of the projective line over $\bm{F}_5$, it acts naturally on
$\bm{F}_5\cup \{\infty\}$. For examples, the fractional linear
transformations $x\mapsto x+1$, $x\mapsto 2/x$ and $x\mapsto 4/x$
correspond to the permutations $(0,1,2,3,4)$,
$(0,\infty)(1,4)(1,2,4,3)(=(0,\infty)(1,2)(3,4))$ and
$(0,\infty)(1,4)$ respectively.

We rewrite the points 0, 1, 2, 3, 4, $\infty$ as 1, 2, 3, 4, 5, 6.
Thus $G_{13}$ and $G_{14}$ are defined by $G_{13}=\langle
(1,2,3,4,5),(1,6)(2,3)(4,5)\rangle \subset S_6$,
$G_{14}=\langle(1,2,3,4,5),(1,6)(2,5)\rangle \subset S_6$.

Define a group homomorphism $\rho:S_5\to S_6$ by $\rho:(1,2)\mapsto (1,6)(2,3)(4,5)$,
$(2,3)\mapsto (1,5)(2,6)(3,4)$, $(3,4)\mapsto(1,2)(3,6)(4,5)$, $(4,5)\mapsto (1,5)(2,3)(4,6)$.

Note that the group $S_5$ is defined by generators $\{(i,i+1):1\le
i\le 4\}$ with relations $(i,i+1)^2=1$ (for $1\le i\le 4$),
$((i,i+1)(i+1,i+2))^3=1$ (for $1\le i\le 3$),
$((i,i+1)(j,j+1))^2=1$ if $|j-i|\ge 2$. These relations are
preserved by $\{ \rho((i,i+1)):1\le i\le 4\}$. Hence $\rho$ is a
well-defined group homomorphism.

We will show that $\rho(S_5)=G_{13}$ and $\fn{Ker}(\rho)=\{1\}$,
i.e.\ $S_5\simeq G_{13}$ as abstract groups. By the definition of
$\rho$, it is easy to verify that
$\rho((1,2,3,4,5))=(1,2,3,4,5)\in S_6$. Since $S_5=\langle
(1,2,3,4,5),(1,2)\rangle$ and $\rho((1,2,3,4,5))$, $\rho((12))\in
G_{13}$. It follows that $\rho(S_5)\subset G_{13}$. Since
$A_5\not\subset \fn{Ker}(\rho)$, it follows that $\rho$ is
injective and $\rho(S_5)=G_{13}$ because $|S_5|=120=|G_{13}|$.

It is possible to construct an embedding of $S_5$ as a transitive
subgroup of $S_6$ by other methods; see, for examples,
\cite[Section 4]{Di}.

Since $G_{14}=PSL_2(\bm{F}_5)$ is a subgroup of
$G_{13}=PGL_2(\bm{F}_5)$ of index 2, it follows that the
restriction of $\rho$ to $A_5$ gives an isomorphism of $A_5$ to
$G_{14}$.

\bigskip
Step 2. Let $\rho':S_6\to GL_6(k)$ be the natural representation
of $S_6$ where $k$ is any field. Then $\rho'\circ \rho: S_5\to
GL_6(k)$ provides the permutation representation of $S_5$ when it
is embedded in $S_6$ via $\rho$. It follows that $S_5$ acts on
$k(x_1,\ldots,x_6)$ via $\rho'\circ \rho$.

When $\fn{char}k=0$, by checking the character table, we find that
the representation $\rho'\circ \rho$ decomposes into $\bm{1}\oplus
\rho_0$ where $\bm{1}$ is the trivial representation of $S_5$, and
$\rho_0$ is the 5-dimensional irreducible representation of $S_5$
which is equivalent to the representation $W'$ in \cite[page
28]{FH}.

\bigskip
Step 3. For any field $k$, let $S_5$ act on the rational function
field $k(y_1,\ldots,y_5)$ by $\sigma(y_i)=y_{\sigma(i)}$ for any
$\sigma \in S_5$. Since $G_{13}\simeq S_5$ by Step 1, we may
consider the action of $G_{13}$ (resp.\ $G_{14}$) on
$k(x_1,\ldots,x_6)$ also. Thus $G_{13}$ and $G_{14}$ act on
$k(x_1,\ldots,x_6,y_1,\ldots,y_5)$.

Apply Part (1) of \cite[Theorem 2.1]{KW} to $k(x_1,\ldots,x_6,y_1,\ldots,y_5)^G$ where $G=G_{13}$ or $G_{14}$.
We find that $k(x_1,\ldots,x_6,y_1,\ldots,y_5)^G=k(x_1,\ldots,x_6)^G(t_1,\ldots,t_5)$ where $g(t_i)=t_i$ for all $g\in G$, all $1\le i\le 5$.

On the other hand, apply Part (1) of \cite[Theorem 2.1]{KW} to $k(x_1,\ldots,x_6,y_1,\ldots,y_5)^G$ again with $L=k(y_1,\ldots,y_5)$.
We get $k(x_1,\ldots,x_6,y_1,\ldots,y_5)^G=k(y_1,\ldots,y_5)^G(s_1,\ldots,\allowbreak s_6)$ where $g(s_i)=s_i$ for all $g\in G$,
all $1\le i\le 6$.

Since $k(y_1,\ldots,y_5)^G$ is $k$-rational when $G=G_{13}\simeq S_5$,
and when $G=G_{14}\simeq A_5$ by Maeda's Theorem \cite[Theorem 2.6]{KW},
we find that $k(x_1,\ldots,x_6)^G$ is stably $k$-rational.

\bigskip
Step 4. We will show that $\bm{C}(x_1,\ldots,x_6)^{G_{13}}$ is
$\bm{C}$-rational. Recall a result of Shepherd-Barron \cite{Sh}
that, if $S_5\to GL(V)$ is any irreducible representation  over
$\bm{C}$, then $\bm{C}(V)^G$ is $\bm{C}$-rational.

By Step 2, since the representation $\rho'\circ \rho$ decomposes,
we may write $\bm{C}(x_1,\ldots,x_6)=\bm{C}(t_1,\ldots,t_6)$ where
$g(t_6)=t_6$ for any $g\in G_{13}$, and $G_{13}\simeq S_5$ acts on
$\bigoplus_{1\le i\le 5} \bm{C}\cdot t_i$ irreducibly.

Apply Shepherd-Barron's Theorem \cite{Sh}.
We get $\bm{C}(t_1,\ldots,t_5)^{G_{13}}$ is $\bm{C}$-rational.
Hence $\bm{C}(x_1,\ldots,x_6)^{G_{13}}=\bm{C}(t_1,\ldots,t_6)^{G_{13}}=\bm{C}(t_1,\ldots,t_5)^{G_{13}}(t_6)$ is also $\bm{C}$-rational.

\bigskip
Step 5. We will show that $\bm{C}(x_1,\ldots,x_6)^{G_{14}}$ is
$\bm{C}$-rational.

By Step 1, $G_{14}\simeq A_5$ as abstract groups.

By \cite[page 29]{FH}, $A_5$ has a faithful complex irreducible representation $A_5\to GL(V)$ where $\dim_{\bm{C}} V=3$.
Let $z_1$, $z_2$, $z_3$ be a dual basis of $V$.
Then $\bm{C}(V)^{G_{14}}=\bm{C}(z_1,z_2,z_3)^{G_{14}}=\bm{C}(z_1/z_3,z_2/z_3,z_3)^{G_{14}}$.

Consider $\bm{C}(x_1,\ldots,x_6)^{G_{14}}$. Define $y_0=\sum_{1\le
i\le 6}x_i$, $y_i=x_i-(y_0/6)$. Sine $G_{14}$ permutes
$x_1,\ldots,x_6$, it follows that $G_{14}$ permutes
$y_1,\ldots,y_6$ where $\sum_{1\le i\le 6} y_i=0$. Thus
$\bm{C}(x_1,\ldots,x_6)^{G_{14}}=\bm{C}(y_1,\ldots,y_5)^{G_{14}}(y_0)$.

Since
$\bm{C}(y_1,\ldots,y_5)^{G_{14}}=\bm{C}(y_1/y_5,y_2/y_5,y_3/y_5,y_4/y_5,y_5)^{G_{14}}$
and $g(y_5)=a_g\cdot y_5+b_g$ for some $a_g,b_g\in
\bm{C}(y_1/y_5,\ldots,y_4/y_5)$, we may apply \cite[Theorem
2.2]{KW}. It follows that
$\bm{C}(y_1/y_5,\ldots,y_4/y_5,y_5)^{G_{14}}=\bm{C}(y_1/y_5,\ldots,y_4/y_5)^{G_{14}}(t)$
for some $t$ with $g(t)=t$ for any $g\in G_{14}$. In conclusion,
$\bm{C}(x_1,\ldots,x_6)^{G_{14}}=\bm{C}(y_1/y_5,y_2/y_5,y_3/y_5,y_4/y_5)^{G_{14}}(t,y_0)$.

On the other hand, apply Part (2) of \cite[Theorem 2.1]{KW} to $\bm{C}(y_1/y_5,y_2/y_5,y_3/y_5,\allowbreak y_4/y_5,z_1/z_3,z_2/z_3)^{G_{14}}$.
We find that $\bm{C}(y_1/y_5,\ldots,y_4/y_5,z_1/z_3,z_2/z_3)^{G_{14}}=\bm{C}(y_1/y_5,\ldots,\allowbreak y_4/y_5)^{G_{14}}(t_1,t_2)$ where
$g(t_1)=t_1$, $g(t_2)=t_2$ for any $g\in G$.
Thus $\bm{C}(x_1,\ldots,x_6)^{G_{14}}=\bm{C}(y_1/y_5,y_2/y_5,y_3/y_5,y_4/y_5)^{G_{14}}(t,y_0)\simeq
\bm{C}(y_1/y_5,\ldots,y_4/y_5)^{G_{14}}(t_1,t_2)=\bm{C}(y_1/y_5,\ldots,\allowbreak y_4/y_5,z_1/z_3,z_2/z_3)^{G_{14}}$.

But $G_{14}$ acts faithfully also on $\bm{C}(z_1/z_3,z_2/z_3)$
because $G_{14}\simeq A_5$ is a simple group. Apply Part (2) of
\cite[Theorem 1]{KW} to
$\bm{C}(y_1/y_5,\ldots,y_4/y_5,z_1/z_3,z_2/z_3)^{G_{14}}$ again
with $L=\bm{C}(z_1/z_3,z_2/z_3)^{G_{14}}$. We get
$\bm{C}(y_1/y_5,\ldots,y_4/y_5,z_1/z_3,z_2/z_3)^{G_{14}}=\bm{C}(z_1/z_3,z_2/z_3)^{G_{14}}\allowbreak
(s_1,s_2,s_3,s_4)$ with $g(s_i)=s_i$ for all $g\in G_{14}$, for
all $1\le i\le 4$.

We conclude that $\bm{C}(x_1,\ldots,x_6)^{G_{14}}\simeq \bm{C}(z_1/z_3,z_2/z_3)^{G_{14}}(s_1,s_2,s_3,s_4)$.

By Castelnuovo's Theorem \cite{Za},
$\bm{C}(z_1/z_3,z_2/z_3)^{G_{14}}$ is $\bm{C}$-rational.
Hence $\bm{C}(x_1,\allowbreak \ldots,x_6)^{G_{14}}$ is $\bm{C}$-rational.
\end{proof}

\begin{remark}
In the last paragraph of Step 5, if we use Zariski-Castelnuovo's
Theorem instead of Castelnuovo's original theorem, then we find a
slightly general result as follows. If $k$ is an algebraically
closed field with $\fn{char}k\ne 2, 5$, then
$k(x_1,\ldots,x_6)^{G_{14}}$ is $k$-rational. Note that the
assumption that $\fn{char}k\ne 2, 5$ is added in order to
guarantee the existence of the 3-dimensional irreducible
representation in \cite[page 29]{FH}.
\end{remark}

\begin{proof}[Proof of Theorem \ref{t1.2}] --------------

\medskip
It remains to prove that, for any field $k$, $k(x_1,\ldots,x_6)^G$
is $k$-rational where $G=G_i$ with $1\le i\le 3$ or $5\le i\le 7$.

\medskip
\begin{Case}{1} $G=G_1$ \end{Case}

Since $G_1=\langle (1,2,3,4,5,6)\rangle$, $k(x_1,\ldots,x_6)^{G_1}=k(G_1)$ is $k$-rational by \cite[Theorem 2.8]{KW}.

\bigskip
\begin{Case}{2} $G=G_2=\langle (1,2)(3,4)(5,6),(1,3,5)(2,6,4)\rangle$. \end{Case}

Write $\sigma=(1,3,5)(2,6,4)$, $\tau=(1,2)(3,4)(5,6)$. Then the
actions are given by
\begin{align*}
\sigma &: x_1\mapsto x_3\mapsto x_5\mapsto x_1,~ x_2\mapsto x_6\mapsto x_4\mapsto x_2, \\
\tau &: x_1\leftrightarrow x_2,~ x_3\leftrightarrow x_4,~ x_5\leftrightarrow x_6.
\end{align*}

Define $y_1=x_1/x_2$, $y_2=x_3/x_6$, $y_3=x_5/x_4$. The we get
\begin{align*}
\sigma &: y_1\mapsto y_2\mapsto y_3\mapsto y_1, \\
\tau &: y_1\mapsto 1/y_1,~ y_2\mapsto 1/y_3,~ y_3\mapsto 1/y_2.
\end{align*}

It follows that $k(x_1,\ldots,x_6)^{G_2}=k(y_1,y_2,y_3,x_2,x_4,x_6)^{G_2}$.
Apply Part (1) of \cite[Theorem 1.1]{KW} with $L=k(y_1,y_2,y_3)$.
We find that $k(y_1,y_2,y_3,x_2,x_4,x_6)^{G_2}=k(y_1,\allowbreak y_2,y_3)^{G_2}(t_1,t_2,t_3)$ where $g(t_i)=t_i$ for all $g\in G_2$,
for all $1\le i\le 3$.

Since $G_2$ acts on $k(y_1,y_2,y_3)$ by purely monomial $k$-automorphisms,
we may apply \cite[Theorem 2.5]{KW}.
Hence $k(y_1,y_2,y_3)^{G_2}$ is $k$-rational.

\bigskip
\begin{Case}{3} $G=G_3=\langle (1,2,3,4,5,6),(1,6)(2,5)(3,4)\rangle$. \end{Case}

Write $\sigma=(1,2,3,4,5,6)$, $\tau=(1,6)(2,5)(3,4)$.
Then $\sigma$ and $\tau$ act on $k(x_1,\ldots,x_6)$ by
\begin{align*}
\sigma &: x_1 \mapsto x_2\mapsto x_3\mapsto x_4\mapsto x_5\mapsto x_6\mapsto x_1, \\
\tau &: x_1\leftrightarrow x_6,~ x_2\leftrightarrow x_5,~ x_3\leftrightarrow x_4.
\end{align*}

\medskip
{\it Subcase 3.1} $\fn{char}k\ne 2$.

Define $y_1=x_1-x_4$, $y_2=x_2-x_5$, $y_3=x_3-x_6$, $y_4=x_1+x_4$, $y_5=x_2+x_5$, $y_6=x_3+x_6$.

Then $k(x_1,\ldots,x_6)=k(y_1,\ldots,y_6)$ and
\begin{align*}
\sigma &: y_1\mapsto y_2\mapsto y_3\mapsto -y_1,~ y_4\mapsto y_5\mapsto y_6\mapsto y_4, \\
\tau &: y_1\mapsto -y_3,~ y_2\mapsto -y_2,~ y_3\mapsto -y_1,~ y_4\leftrightarrow y_6,~ y_5\mapsto y_5.
\end{align*}

Apply Part (1) of \cite[Theorem 2.1]{KW}.
We get $k(y_1,\ldots,y_6)^{G_3}=k(y_1,y_2,y_3)^{G_3}(t_1,\allowbreak t_2,t_3)$ where $\sigma(t_i)=\tau(t_i)=t_i$ for $1\le i\le 3$.

Write $k(y_1,y_2,y_3)=k(y_1/y_3,y_2/y_3,y_3)$.
Apply \cite[Theorem 2.2]{KW}.
We get $k(y_1/y_3,\allowbreak y_2/y_3,y_3)^{G_3}=k(y_1/y_3,y_2/y_3)^{G_3}(t)$ where $\sigma(t)=\tau(t)=t$.

Note that $G_3$ acts on $y_1/y_3$ and $y_2/y_3$ by monomial $k$-automorphisms.
By Hajja's Theorem \cite{Ha}, $k(y_1/y_3,y_2/y_3)^{G_3}$ is $k$-rational.

\medskip
{\it Subcase 3.2} $\fn{char}k=2$.

Define $y_1=x_1/(x_1+x_4)$, $y_2=x_2/(x_2+x_5)$, $y_3=x_3/(x_3+x_6)$, $y_4=x_1+x_4$, $y_5=x_2+x_5$, $y_6=x_3+x_6$.
Then $k(x_1,\ldots,x_6)=k(y_1,\ldots,y_6)$ and
\begin{align*}
\sigma &: y_1\mapsto y_2\mapsto y_3\mapsto y_1+1,~ y_4\mapsto y_5\mapsto y_6\mapsto y_4, \\
\tau &: y_1\mapsto y_3+1,~ y_2\mapsto y_2+1,~ y_3\mapsto y_1+1,~ y_4\leftrightarrow y_6,~ y_5\mapsto y_5.
\end{align*}

Apply Part (1) of \cite[Theorem 2.1]{KW}.
We get $k(y_1,\ldots,y_6)^{G_3}=k(y_1,y_2,y_3)^{G_3}(t_1,\allowbreak t_2,t_3)$ where $\sigma(t_i)=\tau(t_i)=t_i$ for $1\le i\le 3$.

Define $z_1=y_1(y_1+1)$, $z_2=y_1+y_2$, $z_3=y_2+y_3$.

Then $k(y_1,y_2,y_3)^{\langle\sigma^3\rangle}=k(z_1,z_2,z_3)$ and
\begin{align*}
\sigma &: z_1\mapsto z_1+z_2^2+z_2,~ z_2\mapsto z_3\mapsto z_2+z_3+1\mapsto z_2, \\
\tau &: z_1\mapsto z_1+z_2^2+z_3^2+z_2+z_3,~ z_2\leftrightarrow z_3.
\end{align*}

Apply \cite[Theorem 2.2]{KW} to $k(z_1,z_2,z_3)^{\langle\sigma,\tau\rangle}$ with $L=k(z_2,z_3)$.
We get $k(z_1,z_2,z_3)^{G_3}\allowbreak =k(z_2,z_3)^{G_3}(t)$ where $\sigma(t)=\tau(t)=t$.

Define $z_4=z_2+z_3+1$. Then $z_2+z_3+z_4=1$ and
$\sigma:z_2\mapsto z_3\mapsto z_4\mapsto z_2$,
$\tau:z_2\leftrightarrow z_3$, $z_4\mapsto z_4$. Thus $\langle
\sigma,\tau\rangle \simeq S_3$ on $k(z_2,z_3,z_4)$ with
$z_2+z_3+z_4=1$.

Define $u=z_2z_3+z_2z_4+z_3z_4=z_2^2+z_2z_3+z_3^2+z_2+z_3$, $v=z_2z_3z_4=z_2^2z_3+z_2z_3^2+z_2z_3$.

Since $k(u,v) \subset k(z_2,z_3)^{G_3}$ and
$[k(z_2,z_3):k(u,v)]\le 6=[k(z_2,z_3):k(z_2,z_3)^{G_3}]$, it
follows that $k(z_2,z_3)^{G_3}=k(u,v)$. Hence
$k(y_1,y_2,y_3)^{G_3}$ is $k$-rational.

\bigskip
\begin{Case}{4} $G=G_5=\langle(1,2,3)(4,5,6),(1,2)(4,5),(1,4)(2,5)\rangle$. \end{Case}

Write $\sigma=(1,2,3)(4,5,6)$, $\tau=(1,2)(4,5)$,
$\lambda_1=(1,4)(2,5)$,
$\lambda_2=\sigma\lambda_1\sigma^{-1}=(2,5)(3,6)$. Note that
$\langle\lambda_1,\lambda_2\rangle\simeq C_2\times C_2$. The
action of $G_5$ is given by
\begin{align*}
\lambda_1 &: x_1\leftrightarrow x_4,~ x_2\leftrightarrow x_5,~ x_3\mapsto x_3,~ x_6\mapsto x_6, \\
\lambda_2 &: x_1\mapsto x_1,~ x_4\mapsto x_4,~ x_2\leftrightarrow x_5,~ x_3\leftrightarrow x_6, \\
\sigma &: x_1\mapsto x_2\mapsto x_3\mapsto x_1,~ x_4\mapsto x_5\mapsto x_6\mapsto x_4, \\
\tau &: x_1\leftrightarrow x_2,~ x_3\mapsto x_3,~ x_4\leftrightarrow x_5,~ x_6\mapsto x_6.
\end{align*}

\medskip
{\it Subcase 4.1} $\fn{char}k\ne 2$.

Define $y_1=x_1-x_4$, $y_2=x_2-x_5$, $y_3=x_3-x_6$, $y_4=x_1+x_4$, $y_5=x_2+x_5$, $y_6=x_3+x_6$.

Then $k(x_1,\ldots,x_6)=k(y_1,\ldots,y_6)$ and
\begin{align*}
\lambda_1 &: y_1\mapsto -y_1,~ y_2\mapsto -y_2,~ y_3\mapsto y_3,~ y_4\mapsto y_4,~ y_5\mapsto y_5,~ y_6\mapsto y_6, \\
\lambda_2 &: y_1\mapsto y_1,~ y_2\mapsto -y_2,~ y_3\mapsto -y_3,~ y_4\mapsto y_4,~ y_5\mapsto y_5,~ y_6\mapsto y_6, \\
\sigma &: y_1\mapsto y_2\mapsto y_3\mapsto y_1,~ y_4\mapsto y_5\mapsto y_6\mapsto y_4, \\
\tau &: y_1\leftrightarrow y_2,~ y_3\mapsto y_3,~ y_4\leftrightarrow y_5,~ y_6\mapsto y_6.
\end{align*}

Apply Part (1) of \cite[Theorem 2.1]{KW}.
We get $k(x_1,\ldots,x_6)^{G_5}=k(y_1,\ldots,y_6)^{G_5}=k(y_1,y_2,y_3)^{G_5} (t_1,t_2,t_3)$ where $g(t_i)=t_i$ for any $g\in G_5$, any $1\le i\le 3$.

Define $z_1=y_2y_3/y_1$, $z_2=y_1y_3/y_2$, $z_3=y_1y_2/y_3$.
It is not difficult to show that $k(y_1,y_2,y_3)^{\langle\lambda_1,\lambda_2\rangle}=k(z_1,z_2,z_3)$ and the actions of $\sigma$ and $\tau$ are given by
\begin{gather}
\begin{split}
\sigma &: z_1\mapsto z_2\mapsto z_3\mapsto z_1, \\
\tau &: z_1\leftrightarrow z_2,~ z_3\mapsto z_3.
\end{split} \label{eq3.1}
\end{gather}

Hence $k(z_1,z_2,z_3)^{\langle\sigma,\tau\rangle}=k(s_1,s_2,s_3)$ is $k$-rational where $s_1$, $s_2$, $s_3$ are the elementary symmetric functions in $z_1$, $z_2$, $z_3$.

\medskip
{\it Subcase 4.2} $\fn{char}k=2$.

Define $y_1=x_1/(x_1+x_4)$, $y_2=x_2/(x_2+x_5)$, $y_3=x_3/(x_3+x_6)$, $y_4=x_1+x_4$, $y_5=x_2+x_5$, $y_6=x_3+x_6$.
Then $k(x_1,\ldots,x_6)=k(y_1,\ldots,y_6)$ and
\begin{align*}
\lambda_1 &: y_1\mapsto y_1+1,~ y_2\mapsto y_2+1,~ y_3\mapsto y_3,~ y_4\mapsto y_4,~ y_5\mapsto y_5,~ y_6\mapsto y_6, \\
\lambda_2 &: y_1\mapsto y_1,~ y_2\mapsto y_2+1,~ y_3\mapsto y_3+1,~ y_4\mapsto y_4,~ y_5\mapsto y_5,~ y_6\mapsto y_6, \\
\sigma &: y_1\mapsto y_2\mapsto y_3\mapsto y_1,~ y_4\mapsto y_5\mapsto y_6\mapsto y_4, \\
\tau &: y_1\leftrightarrow y_2,~ y_3\mapsto y_3,~ y_4\leftrightarrow y_5,~ y_6\mapsto y_6.
\end{align*}

Apply Part (1) of \cite[Theorem 2.1]{KW}.
We get $k(y_1,\ldots,y_6)^{G_5}=k(y_1,y_2,y_3)^{G_5}\allowbreak (t_1,t_2,t_3)$ where $g(t_i)=t_i$ for any $g\in G_5$, any $1\le i\le 3$.

Define $z_1=y_1(y_1+1)$, $z_2=y_2(y_2+1)$, $z_3=y_1+y_2+y_3$.
It is not difficult to verify that $k(y_1,y_2,y_3)^{\langle \lambda_1,\lambda_2\rangle}=k(z_1,z_2,z_3)$ and
\begin{align*}
\sigma &: z_1\mapsto z_2\mapsto z_1+z_2+z_3^2+z_3,~ z_3\mapsto z_3, \\
\tau &: z_1\leftrightarrow z_2,~ z_3\mapsto z_3.
\end{align*}

Define $z_4=z_1+z_2+z_3^2+z_3$.
It follows that $\sigma: z_1\mapsto z_2\mapsto z_4\mapsto z_1$ and $z_1+z_2+z_4\allowbreak =z_3^2+z_3$.
Define $u=z_1z_2+z_1z_4+z_2z_4=z_1^2+z_2^2+z_1z_2+z_1z_3+z_2z_3+z_1z_3^2+z_2z_3^2$,
$v=z_1z_2z_4=z_1^2z_2+z_1z_2^2+z_1z_2z_3+z_1z_2z_3^2$.
It follows that $k(z_1,z_2,z_3)^{\langle\sigma,\tau\rangle}=k(z_3,u,v)$ is $k$-rational.

\bigskip
\begin{Case}{5}
$G=G_6$ or $G_7$, where $G_6=\langle (1,2,3)(4,5,6),(1,5,4,2)\rangle$,
and $G_7=\langle (1,2,3)\allowbreak (4,5,6),(1,4)(2,5)\rangle$.
\end{Case}

Write $\sigma=(1,2,3)(4,5,6)$, $\tau=(1,5,4,2)$,
$\lambda_1=\tau^2=(1,4)(2,5)$,
$\lambda_2=\sigma\lambda_1\sigma^{-1}=(2,5)(3,6)$. Note that
$\langle \lambda_1,\lambda_2\rangle \simeq C_2\times C_2$. Then
$k(x_1,\ldots,x_6)^{G_6}=k(x_1,\ldots,x_6)^{\langle
\lambda_1,\lambda_2,\sigma,\tau\rangle}$,
$k(x_1,\ldots,x_6)^{G_7}=k(x_1,\ldots,x_6)^{\langle\lambda_1,\lambda_2,\sigma\rangle}$
and the actions are given by
\begin{align*}
\sigma &: x_1\mapsto x_2\mapsto x_3\mapsto x_1,~ x_4\mapsto x_5\mapsto x_6\mapsto x_4, \\
\tau &: x_1\mapsto x_5\mapsto x_4\mapsto x_2\mapsto x_1,~ x_3\mapsto x_3,~ x_6\mapsto x_6, \\
\lambda_1 &: x_1\leftrightarrow x_4,~ x_2\leftrightarrow x_5,~ x_3\mapsto x_3,~ x_6\mapsto x_6, \\
\lambda_2 &: x_1\mapsto x_1,~ x_2\leftrightarrow x_5,~ x_3\leftrightarrow x_6,~ x_4\mapsto x_4.
\end{align*}

The proof is similar to the proof of Case 4.

\medskip
{\it Subcase 5.1} $\fn{char}k\ne 2$.

Define $y_1=x_1-x_4$, $y_2=x_2-x_5$, $y_3=x_3-x_6$, $y_4=x_1+x_4$, $y_5=x_2+x_5$, $y_6=x_3+x_6$.
Then we find that
\begin{align*}
\lambda_1 &: y_1\mapsto -y_1,~ y_2\mapsto -y_2,~ y_3\mapsto y_3,~ y_4\mapsto y_4,~ y_5\mapsto y_5,~ y_6\mapsto y_6, \\
\lambda_2 &: y_1\mapsto y_1,~ y_2\mapsto -y_2,~ y_3\mapsto -y_3,~ y_4\mapsto y_4,~ y_5\mapsto y_5,~ y_6\mapsto y_6, \\
\sigma &: y_1\mapsto y_2\mapsto y_3\mapsto y_1,~ y_4\mapsto y_5\mapsto y_6\mapsto y_4, \\
\tau &: y_1\mapsto -y_2,~ y_2\mapsto y_1,~ y_3\mapsto y_3,~ y_4\leftrightarrow y_5,~ y_6\mapsto y_6.
\end{align*}

Apply Part (1) of \cite[Theorem 2.1]{KW}.
It remains to prove that $k(y_1,y_2,y_3)^G$ is $k$-rational where $G=G_6$ or $G_7$.

Define $z_1=y_2y_3/y_1$, $z_2=y_1y_3/y_2$, $z_3=y_1y_2/y_3$.
Then $k(y_1,y_2,y_3)^{\langle\lambda_1,\lambda_2\rangle}=k(z_1,\allowbreak z_2,z_3)$ and
\begin{align*}
\sigma &: z_1\mapsto z_2\mapsto z_3\mapsto z_1, \\
\tau &: z_1\mapsto -z_2,~ z_2\mapsto -z_1,~ z_3\mapsto -z_3.
\end{align*}

It follows that $k(z_1,z_2,z_3)^{\langle\sigma\rangle}=k(C_3)$ is
$k$-rational by \cite[Theorem 2.8]{KW}. Hence
$k(x_1,\ldots,\allowbreak x_6)^{G_7}$ is $k$-rational.

For $k(x_1,\ldots,x_6)^{G_6}$,
note that $k(z_1,z_2,z_3)^{\langle\sigma,\tau\rangle}=k(z_1/z_3,z_2/z_3,z_3)^{\langle\sigma,\tau\rangle}$.
Apply \cite[Theorem 2.2]{KW}.
We have $k(z_1/z_3,z_2/z_3,z_3)^{\langle\sigma,\tau\rangle}=k(z_1/z_3,z_2/z_3)^{\langle\sigma,\tau\rangle}(t)$ where $\sigma(t)=\tau(t)=t$.

On the other hand, in the last part of Subcase 4.1,
we have $k(z_1,z_2,z_3)^{\langle\sigma,\tau\rangle}$ (see Equation \eqref{eq3.1}).
By the same method as above, we have $k(z_1,z_2,z_3)^{\langle \sigma,\tau\rangle}=k(z_1/z_3,\allowbreak z_2/z_3)^{\langle\sigma,\tau\rangle}(s)$ where $\sigma(s)=\tau(s)=s$.

Note that the actions of $\sigma$, $\tau$ on $z_1/z_3$ and $z_2/z_3$ in \eqref{eq3.1} and in the present situation are the same.
Since $k(z_1,z_2,z_3)^{\langle\sigma,\tau\rangle}$ is $k$-rational in Subcase 4.1,
so is $k(z_1,z_2,z_3)^{\langle\sigma,\tau\rangle}$ in the present case.

\medskip
{\it Subcase 5.2} $\fn{char}k=2$.

Define $y_1=x_1/(x_1+x_4)$, $y_2=x_2/(x_2+x_5)$, $y_3=x_3/(x_3+x_6)$, $y_4=x_1+x_4$, $y_5=x_2+x_5$, $y_6=x_3+x_6$.
Then we have
\begin{align*}
\lambda_1 &: y_1\mapsto y_1+1,~ y_2\mapsto y_2+1,~ y_3\mapsto y_3,~ y_4\mapsto y_4,~ y_5\mapsto y_5,~ y_6\mapsto y_6, \\
\lambda_2 &: y_1\mapsto y_1,~ y_2\mapsto y_2+1,~ y_3\mapsto y_3+1,~ y_4\mapsto y_4,~ y_5\mapsto y_5,~ y_6\mapsto y_6, \\
\sigma &: y_1\mapsto y_2\mapsto y_3\mapsto y_1,~ y_4\mapsto y_5\mapsto y_6\mapsto y_4, \\
\tau &: y_1\mapsto y_2+1,~ y_2\mapsto y_1,~ y_3\mapsto y_3,~ y_4\leftrightarrow y_5,~ y_6\mapsto y_6.
\end{align*}

Apply Part (1) of \cite[Theorem 2.1]{KW}.
It remains to prove that $k(y_1,y_2,y_3)^G$ is $k$-rational where $G=G_6$ or $G_7$.

Define $z_1=y_1(y_1+1)$, $z_2=y_2(y_2+1)$, $z_3=y_1+y_2+y_3$. Then
$k(y_1,y_2,y_3)^{\langle\lambda_1,\lambda_2\rangle}=k(z_1,z_2,z_3)$
and
\begin{align*}
\sigma &: z_1\mapsto z_2\mapsto z_1+z_2+z_3^2+z_3,~ z_3\mapsto z_3, \\
\tau &: z_1\leftrightarrow z_2,~ z_3\mapsto z_3+1.
\end{align*}

Define $z_4=z_1+z_3^2+z_3$, $z_5=z_2+z_3^2+z_3$.
Then $k(z_1,z_2,z_3)=k(z_3,z_4,z_5)$ and
\begin{align*}
\sigma &: z_3\mapsto z_3,~ z_4\mapsto z_5\mapsto z_4+z_5, \\
\tau &: z_3\mapsto z_3+1,~ z_4\leftrightarrow z_5.
\end{align*}

Apply \cite[Theorem 2.2]{KW}.
We get $k(z_3,z_4,z_5)=k(z_4,z_5)(t)$ where $\sigma(t)=\tau(t)=t$.
Thus it remains to consider $k(z_4,z_5)^{\langle\sigma\rangle}$ and $k(z_4,z_5)^{\langle\sigma,\tau\rangle}$.

Note that $\langle\sigma,\tau\rangle\simeq S_3$ on $k(z_4,z_5)$.
Let $t_1$, $t_2$, $t_3$ be the elementary symmetric functions of $z_4$, $z_5$ and $z_4+z_5$.
Be aware that $t_1=z_4+z_5+(z_4+z_5)=0$.
It is easy to see that $k(z_4,z_5)^{\langle\sigma,\tau\rangle}=k(t_2,t_3)$ is $k$-rational.
Hence $k(x_1,\ldots,x_6)^{G_6}$ is $k$-rational.

Consider $k(z_4,z_5)^{\langle\sigma\rangle}=k(z_4/z_5,z_5)^{\langle\sigma\rangle}$.
Apply \cite[Theorem 2.2]{KW}.
We get $k(z_4/z_5,\allowbreak z_5)^{\langle\sigma\rangle}=k(z_4/z_5)^{\langle\sigma\rangle}(s)$ where $\sigma(s)=s$.
Since $k(z_4/z_5)^{\langle\sigma\rangle}$ is $k$-rational by L\"uroth's Theorem.
Thus $k(x_1,\ldots,x_6)^{G_7}$ is $k$-rational.
\end{proof}

%\newpage
%----------------------------------------References
\renewcommand{\refname}{\centering{References}}

\end{document}